\documentclass[12pt,reqno]{amsart}

\usepackage[
  hmarginratio={1:1},     
  vmarginratio={1:1},     
  textwidth=400pt,        
  heightrounded,          
]{geometry}
\usepackage{lipsum}

\usepackage[T1]{fontenc}
\usepackage{ae,aecompl}
\usepackage{graphics,cancel,graphicx}
\usepackage{epsfig,psfrag,manfnt}
\usepackage{mathrsfs}
\usepackage{graphicx,mathabx}
\usepackage{bbm,subfigure,rotating,bm,float,mathdots,wasysym,scalerel}

\usepackage[all,cmtip]{xy}
\usepackage{amssymb,amsmath,amsfonts,amsthm,color}

\usepackage{scalerel,stackengine}
\stackMath
\newcommand\reallywidehat[1]{%
\savestack{\tmpbox}{\stretchto{%
  \scaleto{%
    \scalerel*[\widthof{\ensuremath{#1}}]{\kern-.6pt\bigwedge\kern-.6pt}%
    {\rule[-\textheight/2]{1ex}{\textheight}}
  }{\textheight}%
}{0.5ex}}%
\stackon[1pt]{#1}{\tmpbox}%
}

\newcommand{\calF}{\mathcal{F}}

\newcommand{\calT}{\mathcal{T}}

\newcommand{\mC}{\mathbb{C}}

\newcommand{\mN}{\mathbb{N}}

\newcommand{\mR}{\mathbb{R}}

\newcommand{\mT}{\mathbb{T}}

\newcommand{\mZ}{\mathbb{Z}}

\newtheorem{theorem}{Theorem}[section]
\newtheorem{lemma}[theorem]{Lemma}
\newtheorem{corollary}[theorem]{Corollary}
\newtheorem{proposition}[theorem]{Proposition}

\theoremstyle{definition}

\theoremstyle{definition}

\theoremstyle{definition}

\theoremstyle{definition}

\begin{document}



 \title[]{Doubly invariant subspaces of \\
 the Besicovitch space}

\author[]{Amol Sasane}
\address{Department of Mathematics \\London School of Economics\\
     Houghton Street\\ London WC2A 2AE\\ United Kingdom}
\email{A.J.Sasane@lse.ac.uk}

 \maketitle
 
 \begin{abstract}
 A classical result of Norbert Wiener characterises doubly shift-invariant subspaces for square integrable functions on the unit circle with respect to a finite positive Borel measure $\mu$, as being the ranges of the multiplication maps corresponding to the characteristic functions of $\mu$-measurable subsets of the unit circle. An analogue of this result is given for the Besicovitch Hilbert space of `square integrable almost periodic functions'.
 \end{abstract}

 \section{Introduction}
 
 \noindent The aim of this article is to prove a version of the classical result due to N.~Wiener, characterising doubly shift-invariant subspaces (of the Hilbert space square integrable functions on the circle with respect to a finite positive Borel measure), for the Besicovitch Hilbert space. We give the pertinent definitions below.  
 
 First we recall the aforementioned classical result. See e.g. \cite[Thm.11, \S14, Chap.II]{W} or \cite[Thm. 1.2.1, p.8]{Nik} or \cite[Thm.11, \S14, Chap.II]{W}. Let $\mu$ be a finite, nonnegative Borel measure on the unit circle $\mT:=\{z\in \mC: |z|=1\}$, and let $L^2_\mu(\mT)$ be the Hilbert space of all functions $f:\mT\rightarrow \mC$ such that 
 $$
 \|f\|_2^2:=\int_{\;\!\mT} |f(\xi)|^2 \;\!d\mu(\xi)<\infty,
 $$
 with pointwise operations, and the inner product 
$$
\langle f,g\rangle=\int_{\;\!\mT} f(\xi)\;\! \overline{g(\xi)}\;\!d\mu(\xi)
$$
for $f,g\in L^2_\mu(\mT)$. 
Here $\overline{\cdot}$ denotes complex conjugation.  

\goodbreak 

\noindent  For a $\mu$-measurable set, $\mathbf{1}_\sigma$ is the indicator/characteristic function of $\sigma$, i.e., 
$$
 \mathbf{1}_\sigma(w)=\left\{\begin{array}{ll} 1& \textrm{if } w\in \sigma,\\
 0 & \textrm{if } w\in \mT\;\! \setminus \sigma.
 \end{array}\right.
 $$
The multiplication operator $M_z:L^2_\mu(\mT) \rightarrow  L^2_\mu(\mT)$ is given by 
$$
(M_zf)(w)=wf(w) \;\textrm{ for all }w\in \mT,\; f\in L^2_\mu(\mT),
$$
and is called the {\em shift-operator}. A closed subspace $E\subset L^2_\mu(\mT)$ is called {\em doubly invariant} if 
$M_z E\subset E$ and $(M_z)^* E\subset E$. A closed subspace is doubly invariant if and only if $M_z E=E$. The following result gives a characterisation of doubly invariant subspaces of $L^2_\mu(\mT)$:

 \begin{theorem}[N. Wiener]$\;$
 
 \noindent 
 Let $E\subset L^2_\mu(\mT)$ be a closed subspace of $L^2_\mu(\mT)$. Then $M_z E=E$ if and only if there exists a unique measurable set $\sigma \subset \mT$ such that 
 $$
 E=\mathbf{1}_\sigma L^2_\mu(\mT)=\{f\in L^2_\mu(\mT): f=0\;\mu\textrm{-a.e. on }\mT\;\!\setminus \sigma\}.
 $$
 \end{theorem} 
 
 \noindent We will prove a similar result when $L^2_\mu(\mT)$ is replaced by $AP^2$, the Besicovitch Hilbert space. 
 We recall this space and a few of its properties in the following section, before stating and proving our main result in the final section.

 \section{Preliminaries on the Besicovitch space $AP^2$}

 \noindent For $\lambda \in \mR$, let  $e_\lambda:=e^{i\lambda \cdot }\in L^\infty(\mR)$. 
 Let $\calT$ be the space of trigonometric polynomials, i.e., 
 $\calT$ is the linear span of $\{e_\lambda: \lambda \in \mR\}$. 
 The {\em Besicovitch space} $AP^2$ is the completion of $\calT$ with respect to the inner product 
 $$
 \langle p,q\rangle =\lim_{R\rightarrow \infty} \frac{1}{2R}\int_{-R}^R p(x)\;\! \overline{q(x)} \;\!dx,
 $$
 for $p,q\in \calT$, and where $\overline{\cdot}$ denotes complex conjugation. 
 We remark that elements of $AP^2$ are not to be thought of as functions on $\mR$: 
For example, consider 
 the sequence $(q_n)$ in $\calT$, where 
 $$
 q_n(x):=\sum_{k=1}^n \frac{1}{k}e^{i \frac{1}{k} x } \quad (x\in \mR).
$$
Then $(q_n)$ converges to an element of $AP^2$, but $(q_n(x))_{n\in \mN}$  diverges for all 
$x\in \mR$ (see \cite[Remark~5.1.2, p.91]{Par}). Although elements of $AP^2$ may not be functions on $\mR$, they can be identified 
as functions on the Bohr compactifcation $\mR_B$ of $\mR$, and we elaborate on this below.  
We refer the reader to  \cite[\S7.1]{BKS} and the references therein for further details. 

\noindent For a locally compact Abelian group $G$ written additively, the dual group $G^*$ is the set  
 all continuous characters of $G$. Recall that a {\em character of $G$} is a map  $\chi: G\rightarrow \mT$ such that 
$$
\chi(g+h)=\chi(g)\;\! \chi(h) \quad (g,h\in G).
$$
Then $G^*$ becomes an Abelian group with pointwise multplication, but we continue to write the group operation  in $G^*$ also additively, motivated by the special characters 
$$
G=\mR\owns x\stackrel{\theta}{\mapsto}e^{i\theta x}\in \mT,
$$
 when $G=\mR$. So 
the inverse of $\chi\in G^*$ is denoted by $-\chi$. Then $G^*$ is a locally compact Abelian group with the topology 
given by the basis formed by the sets 
$$
U_{g_1,\cdots, g_n;\epsilon} (\chi)
:=\{ \eta \in G^*: |\eta(g_i)-\chi(g_i)|<\epsilon \textrm{ for all } 1\leq i\leq n\},
$$
where $\epsilon>0$, $n\in \mN:=\{1,2,3,\cdots\}$, $g_1,\cdots, g_n\in G$. 

Let $G^*_d$ denote the group $G^*$ with the discrete topology. The dual group $(G^*_d)^*$ of $G^*_d$ is called the {\em Bohr compactification of $G$}. By the Pontryagin duality theorem\footnote{See e.g. \cite[p.189]{Kat}.}, $G$ is the set of all continuous characters of $G^*$, and since $G_B$ is the set of all (continuous or not) characters of $G^*$, $G$ can be considered to be contained in $G_B$. It can be shown that $G$ is dense in $G_B$. 
Let $\mu$ be the normalised Haar measure in $G_B$, that is, $\mu$ is a positive regular Borel measure such that 

\vspace{0.1cm}

\noindent $\;\bullet\;$  (invariance) $\mu(U)=\mu(U+\xi)$ for all Borel sets $U\subset G_B$, and 
all $\xi\in G_B$,

\vspace{0.05cm}

\noindent $\;\bullet\;$ (normalisation)  $\mu(G_B)=1$. 

\vspace{0.1cm}

\noindent 
 Let $\mR_B=(R^*_d)^*$ denote the Bohr compactification of $\mR$. 
  Let $\mu$ be the normalised Haar measure on $\mR_B$. 
  Let $L^2_{\mu}(\mR_B)$ be the Hilbert space of all functions $f:\mR_B\rightarrow \mC$ such that 
 $$
 \|f\|_2^2:=\int_{\mR_B} |f(\xi)|^2 \;\!d\mu(\xi)<\infty,
 $$
 with pointwise operations and the inner product 
$$
\langle f,g\rangle=\int_{\mR_B} f(\xi)\;\! \overline{g(\xi)}\;\!d\mu(\xi)
$$
for all $f,g\in L^2_{\mu}(\mR_B)$. 
 The Besicovitch space $AP^2$ can be identified as a Hilbert space with $L^2_{\mu}(\mR_B)$, and let  $\iota: AP^2\rightarrow L^2_{\mu}(\mR_B)$ be the Hilbert space isomorphism. 
 Let $L^\infty_{\mu}(\mR_B)$  be the space of $\mu_{B}$-measurable functions that are essentially bounded (that is bounded on $\mR_B$ up to a set of measure $0$) 
 with the essential supremum norm
 $$
 \|f\|_\infty:= \inf\{M\geq 0: |f(\xi)|\leq M \textrm{ a.e.}\}.
 $$
 For an element  $f\in L^\infty_{\mu}(\mR_B)$, let $M_f :L^2_{\mu}(\mR_B)\rightarrow L^2_{\mu}(\mR_B)$ be the multiplication map 
 $\varphi\mapsto f\varphi$, where $f\varphi$ is the pointwise multiplication of $f$ and $\varphi$ as functions on $\mR_B$.  
 
 Let $AP\subset L^\infty(\mR)$ be the $C^*$-algebra of 
 almost periodic functions, namely the closure in $L^\infty(\mR)$ of the space $\calT$ of trigonometric polynomials.    
  Then it can be shown that $AP\subset AP^2$, and $\iota(AP)=C(\mR_B)\subset L^\infty_{\mu}(\mR_B)$. Also, for $f\in AP$,
  $$
\|f\|_2=  \|\iota f\|_2\leq \|\iota f\|_\infty=\|f\|_\infty.
  $$
 For $f,g\in AP$, and $\lambda \in \mR$, 
 $$
 \iota(fg)=(\iota f)(\iota g), \quad \iota(e_0)=\mathbf{1}_{\mR_B}, 
 \quad \overline{\iota(e_\lambda)}=
 \iota(\overline{e_\lambda})=\iota(e_{-\lambda}).
 $$ 
 Every element $f\in AP$ gives a multiplication map, $M_{\iota(f)}$ on $ L^2_{\mu}(\mR_B)$. 
 
 For $f\in AP^2$, the {\em mean value} 
 $$
 {\mathbf{m}}(f):=\int_{\mR_B} (\iota(f))(\xi) \;\!d\mu(\xi)=\langle \iota f, \iota \;\!e_0\rangle= \langle f, e_0\rangle
 $$ 
 exists. The set 
  $$
 \Sigma(f):=\{\lambda \in \mR: {\mathbf{m}}(e_{-\lambda} f)\neq 0\}
 $$ 
 is called the {\em Bohr spectrum of $f$}, and can be shown to be at most countable. We have a Hilbert space isomorphism, via the Fourier transform, between $L^2(\mT)$ and $\ell^2(\mZ)$:
 $$
 L^2(\mT)\owns f\mapsto  (\widehat{f}(n):=\langle f, e^{-int}\rangle)_{n\in \mZ}\in \ell^2(\mZ) .
 $$
  Analogously, we have a representation of $AP^2$ via the Bohr transform. We elaborate on this below. 
 
 Let $\ell^2(\mR)$ be the set of all $f:\mR\rightarrow \mC$ for which the set $\{\lambda \in \mC: f(\lambda)\neq 0\}$ is countable and 
 $$
 \|f\|_2^2:=\sum_{\lambda\;\!\in\;\! \mR} |f(\lambda)|^2  <\infty.
 $$
 Then $\ell^2(\mR)$ is a Hilbert space with pointwise operations and the inner product 
 $$
 \langle f,g\rangle =\sum_{\lambda\;\! \in \;\!\mR} f(\lambda)\;\! \overline{g(\lambda)} .
 $$
 For $\lambda \in \mR$, define the {\em shift-operator} $S_\lambda:\ell^2(\mR)\rightarrow \ell^2(\mR)$ by 
 $$
 (S_\lambda f)(\cdot)=f(\cdot-\lambda).
 $$
 Let $c_{00}(\mR)\subset \ell^2(\mR)$ be the set of finitely supported functions. Define the map 
 $\calF: c_{00}(\mR)\rightarrow AP^2$ as follows: For $f\in c_{00}(\mR)$, 
 $$
 (\calF f)(x)=\sum_{\lambda \;\!\in\;\! \mR} f(\lambda) \;\!e^{i\lambda x}\quad (x\in \mR).
 $$
 By continuity,  $\calF: c_{00}(\mR)\rightarrow AP^2$ can be extended to a map (denoted by the same symbol) 
 $\calF:\ell^2(\mR)\rightarrow AP^2$ defined on all of $\ell^2(\mR)$, and is called the {\em Bohr transform}. The map $\calF :\ell^2(\mR)\rightarrow AP^2$ is a Hilbert space isomorphism.  
 The inverse Bohr transform $\calF^{-1}: AP^2\rightarrow \ell^2(\mR)$ is given by 
 $$
 (\calF^{-1} f)(\lambda)={\mathbf{m}}(f e_{-\lambda}) \quad (\lambda \in \mR).
 $$
 For $\lambda \in \mR$ and $f\in L^2_{\mu}(\mR_B)$, we have 
 the following equality in $\ell^2(\mR)$:
 $$
\calF^{-1} \iota^{-1}(M_{\iota(e_\lambda)} f)
=
(\calF^{-1}\iota^{-1}f)(\cdot-\lambda) 
= 
S_\lambda (\calF^{-1}\iota^{-1}f).
$$
We also note that by the Cauchy-Schwarz inequality in $L^2_{\mu}(\mR_B)$, 
for all functions $f,g\in L^2_\mu(\mR_B)$, we have
\begin{eqnarray*}
\Big(\sum_{\lambda\;\! \in \;\!\mR} (\calF^{-1} \iota^{-1} f)(\lambda)\;\!\overline{(\calF^{-1} \iota^{-1} g)(\lambda)}\Big)^2 
\!\!\!\!&\leq&\!\!\!\! \sum_{\alpha\;\! \in \;\!\mR} |(\calF^{-1} \iota^{-1} f)(\alpha)|^2
\sum_{\beta\;\!\in \;\!\mR} |(\calF^{-1} \iota^{-1} g)(\beta)|^2
\\
\!\!\!\!&=&\!\!\!\!\|f\|_2^2\;\!\|g\|_2^2.
\end{eqnarray*}
We will need the following approximation result (see e.g. \cite{Cor} or \cite{BKS}): 

\begin{proposition}
Let $f\in AP$ and $\Sigma(f)$ be its Bohr spectrum. Then there exists a sequence $(p_n)_{n\in \mN}$ in $\calT$  such that 
\begin{itemize}
\item for all $n\in \mN$, $\Sigma(p_n)\subset \Sigma(f)$, and 
\item $ (p_n)_{n\in \mN}$ converges uniformly to $f$ on $\mR$. 
\end{itemize}
\end{proposition}

\noindent Analogous to the classical Fourier theory where the Fourier coefficients of the pointwise multiplication of sufficiently regular functions is given by the convolution of their Fourier coefficients, we have the following. 

\begin{lemma}
\label{lemma_8_may_2021_15:17}
Let $f\in L^\infty_{\mu}(\mR_B)$ and $g\in L^2_{\mu}(\mR_B)$.  Then for all $\lambda\in \mR$,
$$
(\calF^{-1}\iota^{-1}(fg))(\lambda)=\sum_{\alpha\in \mR}(\calF^{-1}\iota^{-1}f)(\alpha) \;\!
(\calF^{-1} \iota^{-1}g)(\lambda-\alpha).
$$
\end{lemma}
\begin{proof} We first show this for $f,g\in \iota \calT$, and then use a continuity argument 
using the density of $\calT$ in $AP^2$. For $f,g\in \iota \calT$, we have $\Sigma(\iota^{-1}f),\Sigma(\iota^{-1}g)$ are finite subsets of $\mR$, and 
\begin{eqnarray*}
\iota^{-1} f =\sum_{\alpha \;\!\in\;\! \mR}
(\calF^{-1} \iota^{-1} f)(\alpha) \;\!e_\alpha,
&\quad &
\iota^{-1} g =\sum_{\beta\;\! \in\;\! \mR}
(\calF^{-1} \iota^{-1} g)(\beta) \;\! e_\beta.
\end{eqnarray*}
So 
\begin{eqnarray*}
\iota^{-1} (fg)
&=& (\iota^{-1} f)\;\! \iota^{-1} g
= \Big( \sum_{\alpha\;\!\in\;\!  \mR}
(\calF^{-1} \iota^{-1} f)(\alpha)\;\! e_\alpha\Big)
\sum_{\beta \;\!\in \;\!\mR}
(\calF^{-1} \iota^{-1} g)(\beta) \;\!e_\beta\\
&=&
\sum_{\alpha \;\!\in\;\! \mR}
\sum_{\beta \;\!\in\;\! \mR}
(\calF^{-1} \iota^{-1} f)(\alpha)
\;\!
(\calF^{-1} \iota^{-1} g)(\beta)
\;\!e_{\alpha+\beta}.
\end{eqnarray*}
We have 
$$
{\mathbf{m}}(e_a)=\left\{ \begin{array}{ll} 
\langle \iota e_0,\iota e_0\rangle=1 & \textrm{if }a=0,\\
\langle \iota e_a,\iota e_0\rangle=0 & \textrm{if }a\neq 0.
\end{array}\right.
$$
Thus 
\begin{eqnarray*}
(\calF^{-1} \iota^{-1} (fg))(\lambda)
 &=& {\mathbf{m}}(\iota^{-1}(fg)\;\!e_{-\lambda} )
 \phantom{ \sum_{\alpha \in \mR}}
\\
&=& \sum_{\alpha\;\! \in \;\!\mR}
\sum_{\beta\;\! \in \;\!\mR}
(\calF^{-1} \iota^{-1} f)(\alpha)
\;\!
(\calF^{-1} \iota^{-1} g)(\beta)\;\!
{\mathbf{m}}(e_{\alpha+\beta-\lambda}) 
\\
 &=& 
\sum_{\alpha\;\! \in\;\! \mR}
(\calF^{-1} \iota^{-1} f)(\alpha)
\;\!
(\calF^{-1} \iota^{-1} g)(\lambda-\alpha) .
\end{eqnarray*}
Now consider the general case when $f\in L^\infty_{\mu}(\mR_B)$ and $g\in L^2_{\mu}(\mR_B)$. 
Then we can find sequences $(f_n),(g_n)$ in $\iota \calT$ such that 
\begin{itemize}
\item $( f_n)_{n\in \mN}$ converges uniformly to $f$, 
\item $( g_n)_{n\in \mN}$ converges to $g$ in $AP^2$, and 
\item for all $n\in \mN$, $\Sigma(\iota^{-1} f_n)\subset \Sigma(\iota^{-1}f)$, and $\Sigma(\iota^{-1}g_n)\subset \Sigma(\iota^{-1}g)$.
\end{itemize}
We remark that the $g_n$ can be constructed by simply `truncating' the `Bohr series' of $\iota^{-1}g$, since 
$$
\sup_{\substack{F\;\!\subset\;\! \Sigma(\iota^{-1} g)\\ F\textrm{ finite}} }\;
\sum_{\beta\;\!\in\;\! F} |(\calF^{-1} \iota^{-1} g)(\beta)|^2=\|g\|_2^2.
$$
Then, with $\widehat{\cdot}:= \calF^{-1} \iota^{-1}$, we have 
\begin{eqnarray*}
&&\Big| \sum_{\alpha\;\!\in\;\! \mR} (\widehat{f_n})(\alpha) \;\!(\widehat{g_n})(\lambda-\alpha) 
-\widehat{f}(\alpha)\;\!\widehat{g}(\lambda-\alpha)\Big|\\
&\leq & \sum_{\alpha\;\!\in\;\! \mR}
|\widehat{f_n}(\alpha)| |\widehat{g_n}(\lambda -\alpha)-\widehat{g}(\lambda-\alpha)| 
+ 
\sum_{\alpha\;\!\in \;\!\mR}|\widehat{f_n}(\alpha)-\widehat{f}(\alpha)||\widehat{g}(\lambda-\alpha)| \\
&\leq &
\|f_n\|_2 \|g_n-g\|_2+\|f_n-f\|_2 \|g\|_2\phantom{\sum_{\alpha\in \mR}}\textrm{(Cauchy-Schwarz)}
\\
&\leq& 
\|f_n\|_\infty \|g_n-g\|_2+\|f_n-f\|_\infty \|g\|_2
\stackrel{n\rightarrow\infty}{\longrightarrow} 
\|f\|_\infty \cdot 0 +0\cdot \|g\|_2=0.
\end{eqnarray*}
This completes the proof.
\end{proof}

\section{Characterisation of doubly invariant subspaces} 
   
\noindent In this section, we state and prove our main results, namely Theorem~\ref{thm_23_april_2021_18:38} and Corollary~\ref{corollary_20_april_2021_18:15}. Theorem~\ref{thm_23_april_2021_18:38}  is a straightforward adaptation\footnote{It is clear that there is but little novelty in the proof of our Theorem~\ref{thm_23_april_2021_18:38}. It may be argued that all this is implicit in the considerable literature on the subject of doubly invariant subspaces in quite general settings; see notably \cite{Sri}, \cite{HasSri}. Let us then make it explicit!} of the proof of the classical 
version of the theorem given in \cite[Theorem ~1.2.1, p.8]{Nik}.   On the other hand, the main result of the article is Corollary~\ref{corollary_20_april_2021_18:15}, which follows from Theorem~\ref{thm_23_april_2021_18:38} 
by an application of 
Lemma~\ref{lemma_8_may_2021_15:17}. 

 For a measurable set $\sigma\subset \mR_B$, let 
 $\mathbf{1}_\sigma\in L^\infty_{\mu}(\mR_B)$ denote the characteristic function of $\sigma$, i.e.,
$$
 \mathbf{1}_\sigma(\xi)=\left\{\begin{array}{ll} 1& \textrm{if }\; \xi \in \sigma,\\
 0 & \textrm{if }\; \xi\in \mR_B \setminus \sigma.
 \end{array}\right.
 $$

   \goodbreak 
   
 \begin{theorem}
 \label{thm_23_april_2021_18:38}
 Let $E\subset L^2_{\mu}(\mR_B)$ be a closed subspace of $L^2_{\mu}(\mR_B)$.  
 
 \noindent Then the following are equivalent:
 \begin{enumerate}
 \item $M_{\iota(e_\lambda)} E=E$ for all $\lambda \in \mR$.
 \item There exists a unique measurable set $\sigma \subset \mR_B$ such that 
 $$
 E=M_{\mathbf{1}_\sigma} L^2_{\mu}(\mR_B)\\
 =\{ f\in L^2_{\mu}(\mR_B): f=0 \;\mu\textrm{\em-a.e. on }\mR_B\setminus \sigma\}.
 $$
 \end{enumerate}
 \end{theorem}
 \begin{proof} (2)$\Rightarrow$(1): Let $f\in E$. Then there exists an element $\varphi\in L^2_{\mu}(\mR_B)$ such that $f= M_{\mathbf{1}_\sigma}\varphi=\mathbf{1}_\sigma \varphi$. For all $\lambda \in \mR$, $\iota (e_\lambda)\in L^\infty_{\mu}(\mR_B)$, and 
 so 
 $$
 M_{\iota(e_\lambda)} f
 =\iota(e_\lambda)(\mathbf{1}_\sigma \;\!\varphi)
 =(\iota(e_\lambda)\mathbf{1}_\sigma ) \;\!\varphi 
 =\mathbf{1}_\sigma(\iota(e_\lambda)\;\!\varphi)
 =M_{\mathbf{1}_\sigma}\psi,
 $$
 where $\psi:=\iota(e_\lambda)\varphi\in L^2(\mR_B, \mu)$, and so $M_{\iota(e_\lambda)} f\in E$. Thus $M_{\iota(e_\lambda)} E\subset E$. Moreover, 
 \begin{eqnarray*}
 f&=&\mathbf{1}_\sigma \;\!\varphi=(\mathbf{1}_{\mR_B}\mathbf{1}_\sigma) \;\!\varphi
 = (\iota(e_0)\mathbf{1}_\sigma)\;\!\varphi=(\iota(e_{\lambda-\lambda})\mathbf{1}_\sigma)\;\!\varphi
 \\
 &=&\iota(e_\lambda)(\mathbf{1}_\sigma\;\!\iota(e_{-\lambda})\;\!\varphi)=M_{\iota(e_\lambda)} g,
 \end{eqnarray*}
 where $g:=\mathbf{1}_\sigma (\iota(e_{-\lambda})\;\!\varphi)\in M_{\mathbf{1}_\sigma}L^2(\mR_B, \mu)$. Hence $f\in M_{\iota(e_\lambda)} E$. Consequently, $E\subset M_{\iota(e_\lambda)} E$ too.
 
 \medskip 
 
 \noindent (1)$\Rightarrow$(2): Let $P_E:L^2_{\mu}(\mR_B)\rightarrow L^2_{\mu}(\mR_B)$ be the orthogonal projection onto the closed subspace $E$. Set $f=P_E \mathbf{1}_{\mR_B}$. Let $I$ be the identity map on $L^2_{\mu}(\mR_B)$. We claim that 
 $$
 \phantom{AAAAAA}\mathbf{1}_{\mR_B}-f\perp E. \phantom{AAAAAA}\quad \quad \quad (\star)
 $$
 Indeed, for all $g\in E$, 
 \begin{eqnarray*}
 \langle \mathbf{1}_{\mR_B}-f, g\rangle 
& =&
 \langle (I-P_E) \mathbf{1}_{\mR_B}, P_E g\rangle
 =
 \langle (P_E)^*(I-P_E) \mathbf{1}_{\mR_B},  g\rangle
\\
& =&
 \langle P_E(I-P_E) \mathbf{1}_{\mR_B},  g\rangle
 =
 \langle 0,g\rangle=0.
 \end{eqnarray*}
 As $f=P_E \mathbf{1}_{\mR_B}\in E$ and $M_{\iota(e_\lambda)} E=E$ for all $\lambda \in \mR$, 
 we have $\mathbf{1}_{\mR_B}-f \perp M_{\iota(e_\lambda)} f$ for all $\lambda \in \mR$. So for all 
 $p\in \calT$, 
 $$
 \int_{\mR_B} \iota(p) f (\mathbf{1}_{\mR_B}-\overline{f}) \;\!d\mu =0.
 $$
 But $\calT$ is dense in $AP^2$, and $\mu$ is a finite positive Borel measure. So 
 $$
 f(\mathbf{1}_{\mR_B}-\overline{f})=0\;\; \mu{\textrm{-a.e.}}
 $$
 Thus $f=|f|^2$ $\mu$-a.e., so that $f(\xi)\in\{0,1\}$ for all $\xi \in \mR_B$. Set 
 $$
 \sigma=\{\xi\in \mR_B: f(\xi)=1\}.
 $$
 Then $f=\mathbf{1}_{\sigma}$ $\mu$-a.e. As $\mathbf{1}_{\sigma}=f=P_E \mathbf{1}_{\mR_B}\in E$, and 
 as $M_{\iota(e_\lambda)} E=E$ for all $\lambda \in \mR$, it follows that 
 $\mathbf{1}_{\sigma} \;\!\iota(\calT)\subset E$. But $E$ is closed, and thus 
 $$
 \textrm{closure}(\mathbf{1}_{\sigma} \;\!\iota(\calT))\subset E.
 $$
 Since $ \textrm{closure}(\calT)=AP^2$, we conclude that $ \mathbf{1}_{\sigma}L^2_{\mu}(\mR_B)\subset E$. 
 
 Next we want to show that $E\subset \mathbf{1}_{\sigma}L^2_{\mu}(\mR_B)$. 
 To this end, let $g\in E$ be orthogonal to $\mathbf{1}_{\sigma}\;\!L^2_{\mu}(\mR_B)$. In particular, for all $\lambda \in \mR$, 
 $$
 \phantom{AAAAAA}
 \int_{\mR_B} g \;\! \mathbf{1}_{\sigma} \;\!\iota(e_\lambda) \;\!d\mu=0.
  \phantom{AAAAAA} \quad \quad \quad \quad \quad(\asterisk)
 $$
 We want to show that $g=0$. Since $g\in E$, $M_{\iota(e_{\lambda})} g\in E$ for all $\lambda$. 
 So by ($\star$) above, $\mathbf{1}_{\mR_B}-\mathbf{1}_{\sigma} \perp M_{\iota(e_{\lambda})} g$, 
 and noting that $\mathbf{1}_{\mR_B},\mathbf{1}_{\sigma}$ are real-valued, 
 $$
 \phantom{AAAAAa}
  \int_{\mR_B} g\;\!\iota(e_\lambda) (\mathbf{1}_{\mR_B}-\mathbf{1}_{\sigma}) \;\!d\mu =0. 
   \phantom{AAAAA}\quad \quad \quad (\asterisk\asterisk)
 $$
 Hence, using the density of $\iota(\calT)$ in $L^2_{\mu}(\mR_B)$, we obtain from ($\asterisk$) and ($\asterisk \asterisk$) that  
 \begin{eqnarray*}
 g \;\! \mathbf{1}_{\sigma}\!\!&=&\!\!0 \;\;\quad\mu\textrm{-a.e.} \\
  g \;\!( \mathbf{1}_{\mR_B}-\mathbf{1}_{\sigma})\!\!&=&\!\!0 \;\;\quad \mu\textrm{-a.e.}
  \end{eqnarray*}
  Thus $g=g\mathbf{1}_{\mR_B}=0$ $\mu$-a.e., as wanted.
  
  The uniqueness of $\sigma$ up to a set of $\mu$-measure $0$ can be seen as follows: If $E=M_{\mathbf{1}_\sigma} L^2_{\mu}(\mR_B)=M_{\mathbf{1}_{\sigma'}} L^2_{\mu}(\mR_B)$, then 
  taking $\mathbf{1}_B\in L^2_\mu(\mR_B)$, we must have $\mathbf{1}_\sigma=\mathbf{1}_{\sigma'} \varphi$ for some $\varphi \in L^2_\mu(\mR_B)$. So $\sigma\subset \sigma'$. Similarly, $\sigma\subset \sigma'$ as well. 
 \end{proof}
 
 \noindent We now interpret the above characterisation result for doubly invariant subspaces of $AP^2$ in terms of the Bohr coefficients of elements of $E$. Given a measurable set $\sigma \subset \mR_B$, define 
$\widehat{\sigma}\in \ell^2(\mR)$ by 
$$
\widehat{\sigma}(\lambda)=\int_{\mR_B} \mathbf{1}_\sigma \;\!\iota(e_{-\lambda})\;\!d\mu
=\int_\sigma \iota (e^{-i\lambda \cdot}) \;\!d\mu.
$$

 \goodbreak

\begin{corollary}
\label{corollary_20_april_2021_18:15}
 Let $E\subset \ell^2(\mR)$ be a closed subspace of $\ell^2(\mR)$. 
 
 \noindent Then the following are equivalent:
 \begin{enumerate}
 \item $S_\lambda E=E$ for all $\lambda \in \mR$.
 \item There exists a unique measurable set $\sigma \subset \mR_B$ such that 
 \begin{eqnarray*}
 \phantom{Aai}
  E\!\!\!\!&=&\!\!\!\!(\calF^{-1} \iota^{-1} M_{\mathbf{1}_\sigma} \iota\;\!\calF) \;\! \ell^2(\mR)
  \\
 \!\!\!\! &=&\!\!\!\! \Big\{ f:\mR\rightarrow \mC\;\! \Big|\;\! f(\lambda)=\sum_{\alpha \in \mR} \widehat{\sigma}(\lambda-\alpha) \;\!\varphi(\alpha), \;\varphi \in \ell^2(\mR)\Big\}.
\end{eqnarray*}
 \end{enumerate}
 \end{corollary}

\end{document}